\documentclass[11pt]{amsart}

\usepackage[mathscr]{eucal}
\usepackage{amsmath,amssymb,amsfonts,amsthm,enumerate,todonotes}

\newtheorem{theorem}{Theorem}[section]
\newtheorem{lemma}[theorem]{Lemma}
\newtheorem{corollary}[theorem]{Corollary}
\newtheorem{proposition}[theorem]{Proposition}

\newtheorem{theirtheorem}{Theorem}

\newcommand{\Z}{\mathbb Z}
\newcommand{\R}{\mathbb R}

\DeclareMathOperator{\ord}{ord}

\newcommand{\la}{\langle}
\newcommand{\ra}{\rangle}
\newcommand{\be}{\begin{equation}}
\newcommand{\ee}{\end{equation}}
\newcommand{\und}{\;\mbox{ and }\;}
\newcommand{\nn}{\nonumber}
\newcommand{\ber}{\begin{eqnarray}}
\newcommand{\eer}{\end{eqnarray}}
\newcommand{\Sum}[2]{\underset{#1}{\overset{#2}{\sum}}}
\newcommand{\Summ}[1]{\underset{#1}{\sum}}

\newcommand{\wtilde}{\widetilde}

\DeclareSymbolFont{goo}{OMS}{cmsy}{b}{n}
\DeclareMathSymbol{\gooT}{\mathalpha}{goo}{"1}

\begin{document}

\title{A Single Set Improvement to the  $3k-4$ Theorem}

\author{David J. Grynkiewicz}
\email{diambri@hotmail.com}
\address{Department of Mathematical Sciences\\ University of Memphis\\ Memphis, TN 38152, USA}

\begin{abstract}
The $3k-4$ Theorem is a classical result which asserts that if $A,\,B\subseteq \mathbb Z$ are finite, nonempty subsets with
\begin{equation}\label{hyp}|A+B|=|A|+|B|+r\leq |A|+|B|+\min\{|A|,\,|B|\}-3-\delta,\end{equation} where $\delta=1$ if $A$ and $B$ are translates of each other, and otherwise $\delta=0$, then there are arithmetic progressions $P_A$ and $P_B$ of common difference such that $A\subseteq P_A$, $B\subseteq P_B$, $|B|\leq |P_B|+r+1$ and $|P_A|\leq |A|+r+1$. It is one of the few cases in Freiman's Theorem for which exact bounds on the sizes of the progressions are known. The hypothesis \eqref{hyp} is best possible in the sense that there are examples of sumsets $A+B$ having cardinality just one more than that of \eqref{hyp},  yet $A$ and $B$ cannot  \emph{both} be contained in short length arithmetic progressions. In this paper, we show that the hypothesis \eqref{hyp} can be significantly weakened and still yield the same conclusion for \emph{one} of the sets $A$ and $B$. Specifically, if $|B|\geq 3$, $s\geq 1$ is the unique integer with
$$(s-1)s\left(\frac{|B|}{2}-1\right)+s-1<|A|\leq s(s+1)\left(\frac{|B|}{2}-1\right)+s,$$
 and  \begin{equation}\label{hyp2}  |A+B|=|A|+|B|+r< (\frac{|A|}{s}+\frac{|B|}{2}-1)(s+1),\end{equation} then we show  there is an arithmetic progression $P_B\subseteq \mathbb Z$ with $B\subseteq P_B$ and $|P_B|\leq |B|+r+1$. The hypothesis \eqref{hyp2} is best possible (without additional assumptions on $A$) for obtaining such a conclusion.
 \end{abstract}

\maketitle

\section{Introduction}

For finite, nonempty subsets $A$ and $B$ of an abelian group $G$, we define their sumset to be $$A+B=\{a+b:\;a\in A,\,b\in B\}.$$ All intervals will be discrete, so $[x,y]=\{z\in\Z:\; x\leq z\leq y\}$ for real numbers $x,\,y\in \R$. More generally, for $d\in G$ and $x,\,y\in \Z$, we let $$[x,y]_d=\{xd,(x+1)d,\ldots,yd\}$$ denote the corresponding interval with difference $d$.
For a nonempty subset $X\subseteq \Z$, we let $\gcd(X)$ denote the greatest common divisor of all elements of $X$, and use the abbreviation $\gcd^*(X):=\gcd(X-X)$ to denote the affine (translation invariant) greatest common divisor of the set $X$, which is equal to $\gcd(-x+X)$ for any $x\in X$. Note $\gcd^*(X)=\gcd(X)$ when $0\in X$.

The study of the structure of $A$ and $B$ assuming $|A+B|$ is small in comparison to the cardinalities $|A|$ and $|B|$ is an important topic in Inverse Additive Number Theory. For instance, if $A=B\subseteq \Z$ with $|A+A|\leq C|A|$, where $C$ is a fixed constant, then Freiman's Theorem asserts that there is a multi-dimensional progression $P_A\subseteq \Z$ with $A\subseteq  P_A$ and $|P_A|\leq f(C)|A|$, where $f(C)$ is a constant that depends only on $C$. The reader is directed to the text \cite{Tao-vu-book} for a fuller discussion of this result, its generalizations, and its implications and importance.

In this paper, we are interested in the special case of Freiman's Theorem when $|A+B|$ is very small, with $C<3$. The following is the (Freiman) $3k-4$ Theorem, proved in the case $A=B$ by Freiman \cite{freiman-3k4} \cite{freiman-book}, extended (in various forms) to general summands $A\neq B$  by Freiman \cite{freiman-3k4-distinct}, by Lev and Smeliansky  \cite{lev-smel-3k4}, and by Stanchescu \cite{Stanchesuc-3k4},  with  the additional conclusion regarding a long length arithmetic progression added later by  Freiman \cite{freiman-3k4-longAP} (in the special case $A=B$)  and by Bardaji and Grynkiewicz \cite{itziar-3k4} (for general  $A\neq B$). The formulation given below is an equivalent simplification of that given in the text \cite[Theorem 7.1(i)]{Grynk-book}.

\begin{theirtheorem}[$3k-4$ Theorem]
Let $A,\,B\subseteq \Z$ be finite, nonempty subsets with
$$|A+B|=|A|+|B|+r\leq |A|+|B|+\min\{|A|,\,|B|\}-3-\delta,$$ where $\delta=1$ if $A$ and $B$ are translates of each other, and otherwise $\delta=0$.
Then there exist arithmetic progressions $P_A, P_B, P_C\subseteq \Z$, each with common difference $d=\gcd^*(A+B)$, such that $A\subseteq P_A$, $B\subseteq P_B$,  and $C\subseteq A+B$ with
$$|P_A|\leq |A|+r+1,\quad |P_B|\leq |B|+r+1\,\quad \und\quad |P_C|\geq |A|+|B|-1.$$
\end{theirtheorem}

The bounds $|P_A|\leq |A|+r+1$,  $|P_B|\leq |B|+r+1$ and $|P_C|\geq |A|+|B|-1$ are best possible, as seen by the example $A=[0,r]_2\cup [2r+2,|A|+r]$ and $B=[0,r]_2\cup [2r+2,|B|+r]$, which has $A+B=[0,r]_2\cup [2r+2,|A|+|B|+2r]$ for $-1\leq r\leq \min\{|A|,\,|B|\}-3$, showing that all three bounds can hold with equality simultaneously. The bound $|A+A|\leq 3|A|-4$ is tight, as seen by the example $A=[0,|A|-2]\cup \{N\}$ for $N$ large, which shows $|P_A|$ cannot be bounded when $|A+A|\geq 3|A|-3$. Likewise, when $A$ and $B$ are not translates of each other, then the bound $|A+B|\leq |A|+|B|+\min\{|A|,\,|B|\}-3$ is also tight, as seen by the example $B=[0,|B|-1]$ and $A=[0,|A|-2]\cup \{N\}$ for $N$ large and $|A|\geq |B|$.

When $|B|$ is significantly smaller than $|A|$, the hypothesis $|A+B|\leq |A|+2|B|-3$ is rather strong, making effective use of the $3k-4$ Theorem more restricted.
There has only been limited success in obtaining conclusions similar to the $3k-4$ Theorem above the threshold  $|A|+|B|+\min\{|A|,\,|B|\}-3-\delta$. See for instance \cite{huicochea-3k4}, where a weaker bound on $|P_B|$ is obtained under an alternative hypothesis (discussed in the concluding remarks) than our hypothesis \eqref{hyp-bounds}. For versions involving more than two summands, see \cite{huicochea-levconj} \cite{lev-mult-3k4-adem} \cite{lev-mult-3k4}. Some related results may also be found in \cite{chen-3k4} \cite{jin} \cite{lev-long} \cite{Ruzsa-3k4}.

As the previous examples  show, if one wishes to consider sumsets with cardinality above the threshold $|A|+|B|+\min\{|A|,\,|B|\}-3-\delta$, then $A$ and $B$ cannot \emph{both} be contained in short arithmetic progressions.
The goal of this paper is to show that, nonetheless, at least \emph{one} of the sets $A$ and $B$ can, indeed, be contained in a short  arithmetic progression under a much weaker hypothesis than that of the $3k-4$ Theorem. Specifically, our main result is the following theorem, whose bounds are optimal in the sense described afterwards.

\begin{theorem}\label{thm-3k-4-minimprove}
Let $A,\,B\subseteq \Z$ be finite, nonempty subsets with $|B|\geq 3$ and  let $s\geq 1$ be the unique integer with
\be\label{hyp-bounds}(s-1)s\left(\frac{|B|}{2}-1\right)+s-1<|A|\leq s(s+1)\left(\frac{|B|}{2}-1\right)+s.\ee
 Suppose \be\label{hyp1a} |A+B|=|A|+|B|+r< (\frac{|A|}{s}+\frac{|B|}{2}-1)(s+1).\ee
Then there exists an arithmetic progression $P_B\subseteq \Z$ such that $B\subseteq P_B$ and $|P_B|\leq |B|+r+1$.
\end{theorem}

The hypothesis \eqref{hyp-bounds} depends on the relative size of $|A|$ and $|B|$. This dependence is necessary, and essentially best possible, as seen by the example $B=[0,\frac{|B|}{2}-1]\cup (N+[0,\frac{|B|}{2}-1])$ and $A=[0,\frac{|A|}{s}-1]\cup (N+[0,\frac{|A|}{s}-1])\cup (2N+[0,\frac{|A|}{s}-1])\cup \ldots\cup ((s-1)N+[0,\frac{|A|}{s}-1])$  for $|B|$ even with $s\mid |A|$ and $N$ large. It is then a minimization problem (carried out in Lemma \ref{lemma-2D-redcalc}) that the optimal choice of $s$ depends on the relative size of $|A|$ and $|B|$ as described in  \eqref{hyp-bounds}.  The bound $|P_B|\leq |B|+r+1$ is also best possible, as seen by the example $B=[0,|B|-2]\cup \{|B|+r\}$ and $A=[0,|A|-1]$.
As a weaker consequence of Theorem \ref{thm-3k-4-minimprove}, we  derive the following corollary, which eliminates  the parameter $s$.

\begin{corollary}\label{cor-3k-4-minimprove}
Let $A,\,B\subseteq \Z$ be finite, nonempty subsets.
 Suppose \be\nn |A+B|=|A|+|B|+r<|A|+\frac{|B|}{2}-1+2\sqrt{|A|(\frac{|B|}{2}-1)}.\ee
Then there exists an arithmetic progression $P_B\subseteq \Z$ such that $B\subseteq P_B$ and $|P_B|\leq |B|+r+1$.
\end{corollary}

\section{Preliminaries}

For an abelian group $G$ and nonempty subset $X\subseteq G$, we let $$\mathsf H(X)=\{g\in G:\; g+X=X\}\leq G$$ denote the stabilizer of $X$, which is the largest subgroup $H$ such that $X$ is a union of $H$-cosets. The set $X$ is called \emph{aperiodic} if $\mathsf H(X)$ is trivial, and \emph{periodic} if $\mathsf H$ is nontrivial. More specifically, we say $X$ is $H$-periodic if $H\leq \mathsf H(X)$, equivalently, if $X$ is a union of $H$-cosets. For a subgroup $H\leq G$, we let $$\phi_H:G\rightarrow G/H$$ denote the natural homomorphism. We let $\la X\ra$ denote the subgroup generated by $X$, and let $\la X\ra_*=\la X-X\ra$ denote the affine (translation invariant) subgroup generated by $X$, which is the minimal subgroup $H$ such that $X$ is contained in an $H$-coset. Note $\la X\ra_*=\la -x+X\ra$ for any $x\in X$. In particular, $\la X\ra_*=\la X\ra$ when $0\in X$. If $k\in\Z$, then $k\cdot A=\{kx:\;x\in A\}$ denotes the $k$-dilate of $A$.
%The relative complement of $A$ is defined as $$\barr{A}{H}:=(H+A)\setminus A.$$ When the subgroup $H$ is implicit, it will usually be dropped from the notation.

Kneser's Theorem \cite[Theorem 6.1]{Grynk-book} \cite[Theorem 5.5]{Tao-vu-book} is a core result in inverse additive theory.

\begin{theirtheorem}[Kneser's Theorem]\label{thm-kt}
Let $G$ be an abelian group, let $A,\,B\subseteq G$ be finite, nonempty subsets, and let $H=\mathsf H(A+B)$. Then $$|A+B|\geq |A+H|+|B+H|-|H|=|A|+|B|-|H|+\rho,$$ where $\rho=|(A+H)\setminus A|+|(B+H)\setminus H|\geq 0$.
\end{theirtheorem}

%A simple consequence of the Pigeonhole Principle is the  following result \cite[Theorem 5.1(i)]{Grynk-book}

%\begin{theirtheorem}[Pigeonhole Bound]\label{phb} Let $G$ be a finite abelian group and let $A,\,B\subseteq G$ be subsets with $|A|+|B|\geq |G|+1$. Then $A+B=G$.
%\end{theirtheorem}

A very special case of Kneser's Theorem is the following basic bound for integer sumsets.

\begin{theirtheorem}\label{CDT-for-Z}
Let $A,\,B\subseteq \Z$ be finite, nonempty subsets. Then $|A+B|\geq |A|+|B|-1$.
\end{theirtheorem}

If $|A+B|\leq |A|+|B|-1$, then $|\phi_H(A)+\phi_H(B)|=|\phi_H(A)|+|\phi_H(B)|-1$ follows from Kneser's Theorem, where $H=\mathsf H(A+B)$, reducing the description of sumsets with $|A+B|\leq |A|+|B|-1$ to the case when $A+B$ is aperiodic with $|A+B|=|A|+|B|-1$.  The complete description is then addressed by the Kemperman Structure Theorem. We summarize the relevant details here, which may be found in
 \cite[Chapter 9]{Grynk-book} and are summarized in more general form  in \cite{ittI}

Let $A,\,B\subseteq G$ and $H\leq G$. A nonempty subset of the form $(\alpha+H)\cap A$ is called an \emph{$H$-coset slice} of $A$. If $A_\emptyset\subseteq A$  is a nonempty subset of an $H$-coset and $A\setminus A_\emptyset$ is $H$-periodic, then $A_\emptyset$ is an $H$-coset slice and we say that $A_\emptyset$ \emph{induces an  $H$-quasi-periodic decomposition} of $A$, namely,
 $A=(A\setminus A_\emptyset)\cup A_\emptyset$.
If, in addition,   $B_\emptyset
\subseteq B$  induces an  $H$-quasi-periodic decomposition, and $\phi_H(A_\emptyset)+\phi_H(B_\emptyset)$ is a unique expression element in $\phi_H(A)+\phi_H(B)$, then $A_\emptyset+B_\emptyset\subseteq A+B$ also induces an $H$-quasi-periodic decomposition.

Let $X,\,Y\subseteq G$ be finite and nonempty subsets with $K=\la X+Y\ra_*$.  We say that the pair $(X,Y)$ is \emph{elementary of  type} (I), (II), (III) or (IV) if there are $z_A,\,z_B\in G$ such that $X=z_A+A$ and $Y=z_B+B$ for a pair of subsets $A,\,B\subseteq K$ satisfying the corresponding requirement below:
\begin{itemize}
\item[(I)]  $|A|=1$ or $|B|=1$.
\item[(II)] $A$ and $B$ are arithmetic progressions of common difference $d\in K$ with $|A|,\,|B|\geq 2$ and $\ord(d)\geq |A|+|B|-1\geq 3$.
\item[(III)] $|A|+|B|=|K|+1$ and there is precisely one unique expression element in the sumset $A+B$; in particular,   $A+B=K$, \, $|A|,\,|B|\geq 3$, and $|K|\geq 5$.
    \item[(IV)] $B=-(K\setminus A)$ and the sumset $A+B$ is aperiodic and contains no unique expression elements; in particular,   $A+B=A-(K\setminus A)=K\setminus \{0\}$, \ $|A|,\,|B|\geq 3$, and $|K|\geq 7$.
\end{itemize}

We will need the following result regarding type (III) elementary pairs.

\begin{lemma}
\label{thm-typeIII} Let $G$ be an abelian group and let $A,\,B\subseteq G$ be finite, nonempty subsets. Suppose $(A,B)$ is a type (III) elementary pair with $a_0+b_0$ the unique expression element in $A+B$, where $a_0\in A$ and $b_0\in B$. Then $$(A\setminus \{a_0\})+(B\setminus \{b_0\})=(A+B)\setminus \{a_0+b_0\}.$$
\end{lemma}

\begin{proof}
 Without loss of generality, we may assume that
$a_0=b_0=0$ and $G=H$. Let $A'=A\setminus \{0\} $ and
$B'=B\setminus \{ 0\}$. Suppose by contradiction $\{0,g\}\subseteq G\setminus (A'+B')$ with $g\neq 0$. Since $g\in G= A+B$ and $g\notin A'+B'$, it follows that every expression $g=x+y\in A+B$, with $x\in A$ and $y\in B$, must have $x=0$ or $y=0$. As a result, since there are at least \emph{two} such expressions (as $0\in A+B$ is the only unique expression element for the type (III) pair), it follows that are exactly two, namely one of the form $g=0+y$ with $y\in B$, and the other of the form $g=x+0$ with $x\in A$, whence  \be\label{ghostly}g\in A\cap
B.\ee Since $0,\,g\notin A'+B'$, we have $(\{0,g\}-A')\cap B'=\emptyset$, and since $(A,B)$ has type (III), we have $|A'|+|B'|=|A|+|B|-2=|G|-1$. As a result, $|\{0,g\}-A'|\leq |G|-|B'|= |A'|+1$, which is easily seen to only be possible if  $A'=A'_1\cup P_1$, where $A'_1$ is $K$-periodic (or empty), $P_1$ is an arithmetic progression with difference $g$, and $K=\langle g\rangle$; moreover, since $g\in A'$ but $0\notin A'$ (see \eqref{ghostly}), we conclude that the first term in $P_1$ must in fact be $g$. Likewise $B'=B'_1 \cup P_2$ with
$B'_1$ $K$-periodic (or empty) and $P_2$ an arithmetic progression with difference $g$ whose first term is $g$. Thus $0\in P_1+K$ and $0\in P_2+K$. Hence, since $0+0$ is a unique expression element in $A+B$, it follows, in view of $A'=A'_1\cup P_1$ and $B'=B'_1 \cup P_2$, that $0$ is a unique expression element in $\phi_K(A)+\phi_K(B)$. Consequently, any unique expression element from $(P_1\cup \{0\})+(P_2\cup\{0\})$ is also a unique expression element in $A+B$.

Since $g$ is the first term in both $P_1$ and $P_2$, it follows that
$P_1\cup \{0\}$ and $P_2\cup\{0\}$ are both arithmetic progressions with difference $g$. Thus, since $(P_1\cup \{0\})+(P_2\cup\{0\})$ contains a unique expression element, namely $0+0$, it follows that $(P_1\cup \{0\})+(P_2\cup\{0\})$ must contain another unique expression element as well, namely $g_1+g_2$, where $g_1\in P_1$ is the last term of the progression $P_1$ and $g_2\in P_2$ is the last term of the progression $P_2$, contradicting (in view of the previous paragraph) that $0+0$ is the only unique expression element in $A+B$.
\end{proof}

The following is the `dual' formulation of the Kemperman Structure Theorem \cite[Theorem 9.2]{Grynk-book}, introduced by Lev \cite{lev-kemp}.

\begin{theirtheorem}[KST-Dual Form]\label{KST-}
Let $G$ be a nontrivial abelian group and let $A,\,B\subseteq G$ be finite, nonempty subsets. A necessary and sufficient condition for $$|A+B|=|A|+|B|-1,$$ with $A+B$ containing a unique expression element when $A+B$ is periodic, is that either $(A,B)$ is elementary of type (IV)  or else there exists a finite, proper subgroup $H<G$ and nonempty subsets $A_\emptyset \subseteq A$ and $B_\emptyset\subseteq B$ inducing $H$-quasi-periodic decompositions  such that
\begin{itemize}
\item[(i)] $(\phi_H(A),\phi_H(B))$ is elementary of some type (I)--(III),
\item[(ii)] $\phi_H(A_\emptyset)+\phi_H(B_\emptyset)$ is a unique expression element in $\phi_H(A)+\phi_H(B)$,
\item[(iii)] $|A_\emptyset+B_\emptyset|=|A_\emptyset|+|B_\emptyset|-1$, and
\item[(iv)] either $A_\emptyset+B_\emptyset$ is aperiodic or contains a unique expression element.
\end{itemize}
\end{theirtheorem}

If $G$ and $G'$ are abelian groups and $A,\,B\subseteq G$ are finite, nonempty subsets, then a Freiman homomorphism is a map $\psi:A+B\rightarrow G'$, defined by some coordinate maps $\psi_A:A\rightarrow G'$ and $\psi_B:B\rightarrow G'$, such that $\psi(x+y)=\psi_A(x)+\psi_B(y)$ for all $x\in A$ and $y\in B$ is well-defined. The sumset $\psi_A(A)+\psi_B(B)$ is then the homomorphic image of $A+B$ under $\psi$. If $\psi$ is injective on $A+B$, then $\psi$ is a Freiman isomorphism, in which case the sumsets $A+B$ and  $\psi_A(A)+\psi_B(B)$ are isomorphic, denoted $A+B\cong  \psi_A(A)+\psi_B(B)$. See \cite[Chapter 20]{Grynk-book}. Equivalently, if there are coordinate maps $\psi_A:A\rightarrow G'$ and $\psi_B:B\rightarrow G'$ such that $\psi_A(x)+\psi_B(y)=\psi_A(x')+\psi_B(y')$ if and only if $x+y=x'+y'$, for any $x,\,x'\in A$ and $y,\,y'\in B$, then $A+B\cong \psi_A(A)+\psi_B(B)$. Isomorphic sumsets have the same behavior with respect to their sumset irrespective of the ambient group in which they live.

The proof of Theorem \ref{thm-3k-4-minimprove} will involve the use of modular reduction, introduced by Lev and Smeliansky \cite{lev-smel-3k4}, in the more general form developed in \cite[Chapter 7]{Grynk-book}. We summarize the needed details from \cite[Chapter 7]{Grynk-book}.

Suppose $A,\,B\subseteq \Z$ are finite nonempty subsets and $n\geq 2$ is an integer.  Let $\phi_n:\Z\rightarrow \Z/n\Z$ denote the natural homomorphism.
%We can decompose the set  $A=\bigcup_{\alpha\in I}A_\alpha$, where each $A_\alpha=(\alpha+n\Z)\cap A\neq \emptyset$ denotes a disjoint $n\Z$-coset slice of $A$, so $|I|=|\phi_n(A)|$, and each $A_\alpha$ is a maximal by inclusion subset of $A$ consisting of elements congruent to each other modulo $n$. We can likewise decompose $B=\bigcup_{\beta\in J}B_\beta$.
For each $i\geq 0$, let $A_i\subseteq \Z/n\Z$ be the subset consisting of all $x\in\Z/n\Z$ for which there are least $i+1$ elements of $A$ congruent to $x$ modulo $n$. Thus $\phi_n(A)=A_0\supseteq A_1\supseteq A_2\supseteq \ldots$ and $\Summ{i\geq 0}|A_i|=|A|$. Likewise define $B_j$ for each $j\geq 0$, so $\phi_n(B)=B_0\supseteq B_1\supseteq B_2\supseteq \ldots$ and $\Summ{j\geq 0}|B_j|=|B|$. Set $$\wtilde A=\bigcup_{i\geq 0}(A_i\times \{i\})\quad\und\quad \wtilde B=\bigcup_{j\geq 0}(B_j\times\{j\}).$$ Thus $\wtilde A,\,\wtilde B\subseteq \Z/n\Z\times \Z$ with $|\wtilde A|=|A|$ and $|\wtilde B|=|B|$.
Then $\wtilde A+\wtilde B=\bigcup_{k\geq 0}(C_k\times \{k\})$, where $$C_k=\bigcup_{i+j=k}(A_i+B_j)$$ for $k\geq 0$. Thus $\phi_n(A+B)=C_0\supseteq C_1\supseteq C_2\supseteq\ldots$.
Let $G=\Z/n\Z$ and let $H\leq G$ be a subgroup. Consider an arbitrary $z\in G/H$, say corresponding to the coset $z'+H$. Let $k_z\geq 0$ be the maximal integer such that $z'+H\subseteq C_{k_z}$, or else set $k_z=-1$ if $z'+H\nsubseteq C_k$ for all $k\geq 0$. Set \begin{multline*}\delta_z=\max\Big(\{0\}\cup \big\{|(x+H)\cap A_i|+|(y+H)\cap B_j|-1-|H|-|(z+H)\cap C_{k_z+1}|:\\
 i+j=k_z,\; \phi_H(x)+\phi_H(y)=z\big\}\Big)\geq 0.\end{multline*}
Then \cite[Corollary 7.1]{Grynk-book} shows that $\wtilde A+\wtilde B$ can be used to estimate the size of $|A+B|$ as follows.

\begin{theirtheorem}
\label{modular-red-cor}
Let $A,\,B\subseteq \Z$ be finite, nonempty sets, let $n\geq 2$ be an integer, and let all notation be as above. Then
$$|A+B|\geq |\wtilde A+\wtilde B|+\Summ{z\in G/H}\delta_z.$$
\end{theirtheorem}

We will use the above machinery in the case when $\min B=0$ and $n=\max B$. In such case,
$A_t\subseteq \ldots\subseteq A_0=\phi_n(A)\subseteq \Z/n\Z$, where $t\geq 0$ is the maximal index such that  $A_t\neq \emptyset$,
$\{0\}=B_1\subseteq B_0=\phi_n(B)\subseteq \Z/n\Z$ and $C_{t+1}\subseteq \ldots\subseteq C_0=\phi_n(A+B)\subseteq \Z/n\Z$,   $$|B_0|=|B|-1,\quad\und\quad \Sum{i=0}{t}|A_i|=|A|.$$
Now $\wtilde A+\wtilde B=\bigcup_{i=0}^{t+1}(C_i\times \{i\})$ with
 $C_0=A_0+B_0$, \ $C_{t+1}=A_t+B_1=A_t$ and $$C_i=(A_i+B_0)\cup (A_{i-1}+B_1)=(A_i+B_0)\cup A_{i-1}\quad\mbox{ for $i\in [1,t]$}.$$
If $H\leq G=\Z/n\Z$ is a subgroup, and $z\in (G/H)\setminus \phi_H(A_0)$, then set
$$\delta'_z=\max\Big(\{0\}\cup \big\{|(x+H)\cap A_0|+|(y+H)\cap B_0|-1-|H|:\;
  \phi_H(x)+\phi_H(y)=z\big\}\Big)\geq 0.$$
 As a special case of Theorem \ref{modular-red-cor}, we obtain the following corollary.

 \begin{corollary}\label{cor-modred}
 Let $A,\,B\subseteq \Z$ be finite, nonempty sets with $0=\min B$ and $n=\max B\geq 2$,  and let all notation be as above. Then
$$|A+B|\geq |A_0+B_0|+|A|+\underset{z\notin \phi_H(A_0)}{\Summ{z\in G/H}}\delta'_z.$$
 \end{corollary}

\begin{proof} For $z\in G/H$, let $c_z=|(z'+H)\cap C_{1}|$, where $z$ corresponds to the coset $z'+H$. Recall that $B_1=\{0\}$. Then,
by Theorem \ref{modular-red-cor}, we have \begin{align}\nn |A+B|&\geq |\wtilde A+\wtilde B|+\underset{z\notin \phi_H(A_0)}{\Summ{z\in G/H}}\delta_z\geq |A_0+B_0|+\Sum{i=0}{t}|A_i+B_1|+\underset{z\notin \phi_H(A_0)}{\Summ{z\in G/H}}c_z+\underset{z\notin \phi_H(A_0)}{\Summ{z\in G/H}}\delta_z\\\label{teel}
&=|A_0+B_0|+\Sum{i=0}{t}|A_i|+\underset{z\notin \phi_H(A_0)}{\Summ{z\in G/H}}(c_z+\delta_z)=|A_0+B_0|+|A|+\underset{z\notin \phi_H(A_0)}{\Summ{z\in G/H}}(c_z+\delta_z).
\end{align}
Consider an arbitrary $z\in G/H$ with $z\notin \phi_H(A_0)$. If $k_z\geq 1$, then  $c_{z}=|H|>\delta'_z$, with the inequality holding trivially by definition of $\delta'_z$, and the equality following from the definitions of $k_z$ and $c_z$.  Otherwise, it follows from the definitions involved that $c_z+\delta_z\geq \delta'_z$. Regardless, we find $\underset{z\notin \phi_H(A_0)}{\Summ{z\in G/H}}(c_z+\delta_z)\geq \underset{z\notin \phi_H(A_0)}{\Summ{z\in G/H}}\delta'_z$, which combined with \eqref{teel} yields the desired lower bound.
\end{proof}

The idea of using compression to estimate sumsets in higher dimensional spaces is a classical technique. See \cite[Section 7.3]{Grynk-book}. We outline briefly what we will need. Let $A,\,B\subseteq \R^2$ be finite, nonempty subsets. Let $x,\,y\in \R^2$ be a basis for $\R^2$. We can decompose $A=\bigcup_{\alpha\in I}A_\alpha$, where each $A_\alpha=(\alpha+\R x)\cap A\neq \emptyset$.  Then $|I|$ equals the number of lines parallel to the line $\R x$ that intersect $A$. We can likewise decompose $B=\bigcup_{\beta\in J}B_\beta$. The linear compression (with respect to $x$) of $A$ is the set $\mathsf C_{x,y}(A)$ obtained by taking $A$ and replacing the elements from each $A_\alpha$ by the arithmetic progression with difference $x$ and length $|A_\alpha|$ contained in $\alpha+\R x$ whose first term lies on the line $\R y$. We likewise define $\mathsf C_{x,y}(B)$. A simply argument (see \cite[eq. (7.18)]{Grynk-book}) shows
$$|A+B|\geq |\mathsf C_{x,y}(A)+\mathsf C_{x,y}(B)|.$$
Finally, we will need the following discrete analog of the Brunn-Minkowski Theorem for  two-dimensional sumsets \cite[Theorem 1.3]{oriol2D} \cite[Theorem 7.3]{Grynk-book}.

\begin{theirtheorem}
\label{2D-brunn-Mink} Let $A,\,B\subseteq \R^2$ be finite, nonempty subsets, let $\ell\subseteq \R^2$ be a line, let $m$ be the number of lines parallel to $\ell$ that intersect $A$, and let $n$ be the number of parallel lines to $\ell$ that intersect $B$. Then
$$|A+B|\geq \Big(\frac{|A|}{m}+\frac{|B|}{m}-1\Big)(m+n-1).$$
\end{theirtheorem}

\section{The Proof}

We begin with a lemma showing that a pair of sets
$A,\,B\subseteq \Z$ being short arithmetic progressions modulo $N$ with common difference forces the sumset $A+B$ to be isomorphic to a two-dimensional sumset from $\Z^2$.

\begin{lemma}\label{Lemma-ap-mod-reduction}
Let $A,\,B\subseteq \Z$ be finite, nonempty subsets, let $N\geq 1$ be an integer, and let $\varphi:\Z\rightarrow \Z/N\Z$ be the natural homomorphism. Suppose $\varphi(A)$ and $\varphi(B)$ are arithmetic progressions with common difference $d\in [1,N-1]$ modulo $N$ such that $|\varphi(A)|+|\varphi(B)|-1\leq \ord(\varphi(d))$.
 Then there is a Freiman isomorphism $$A+B\cong \bigcup_{i=0}^{m-1}(X_i\times\{i\})+\bigcup_{j=0}^{n-1}(Y_j\times \{j\})\subseteq \Z^2,$$ where $A=A_0\cup \ldots\cup A_{m-1}$ and $B=B_0\cup\ldots \cup B_{n-1}$ are the partitions of $A$ and $B$ into distinct residue classes modulo $N$ indexed so that $\varphi(A_i)-\varphi(A_{i-1})=\varphi(B_j)-\varphi(B_{j-1})=\varphi(d)$ for all $i\in [1,m-1]$ and $j\in [1,n-1]$, with $\alpha_0\in A_0$, \ $\beta_0\in B_0$, \  $\alpha_i=\alpha_0+id$, \ $\beta_j=\beta_0+jd$, \ $X_i=\frac{1}{N}\cdot (A_i-\alpha_i)\subseteq \Z$ and $Y_i=\frac{1}{N}\cdot (B_j-\beta_j)\subseteq \Z$, for $i\in [0,m-1]$ and $j\in [0,n-1]$.
\end{lemma}

\begin{proof}
Let $d\in [1,N-1]\subseteq \Z$ be the common difference  modulo $N$ for the arithmetic progressions $\varphi(A)$ and $\varphi(B)$, and let $\alpha_0\in A_0$ and $\beta_0\in B_0$. Set \be\label{def-alphbeta}\alpha_i=\alpha_0+id \quad \und\quad \beta_j= \beta_0+jd,\quad\mbox{ for $i\in [0,m-1]$ and $j\in [0,n-1]$}.\ee
Then each $\alpha_i$ is a representative modulo $N$ for the residue classes $A_i$, and each $\beta_j$
is a representative modulo $N$ for the residue classes $B_j$, for $i\in [0,m-1]$ and $j\in [0,n-1]$.
Note  $$m+n-1=|\varphi(A)|+|\varphi(B)|-1\leq \ord(\varphi(d))$$ by hypothesis. As a result, \be\label{cond-iso}\alpha_i+\beta_j\equiv \alpha_{i'}+\beta_{j'}\mod N\quad\mbox{ if and only if }\quad i+j=i'+j'.\ee
For $i\in [0,m-1]$ and $j\in [0,n-1]$, set $X_i=\frac{1}{N}\cdot (A_i-\alpha_i)\subseteq \Z$ and $Y_i=\frac{1}{N}\cdot (B_j-\beta_j)\subseteq \Z$. Thus $A_i=\alpha_i+N\cdot X_i$ and $B_j=\beta_j+N\cdot Y_j$ for $i\in [0,m-1]$ and $j\in [0,n-1]$. Define the maps $\varphi_A:A\rightarrow \Z^2$ and $\varphi_B:B\rightarrow \Z^2$ by $$\varphi_A(\alpha_i+Nx)=(x,i)\quad\und\quad \varphi_B(\beta_j+Ny)=(y,j),$$ where $x\in X_i$ and $y\in Y_j$. Then $\varphi_A$ and $\varphi_B$ are clearly injective on $A$ and $B$, respectively.

Suppose $(\alpha_i+Nx)+(\beta_j+Ny)=(\alpha_{i'}+Nx')+(\beta_{j'}+Ny')$. Reducing modulo $N$ and applying \eqref{cond-iso}, it follows that $i+j=i'+j'$, in turn implying $\alpha_i+\beta_j=\alpha_{i'}+\beta_{j'}$ per the definitions in \eqref{def-alphbeta}. But now $(\alpha_i+Nx)+(\beta_j+Ny)=(\alpha_{i'}+Nx')+(\beta_{j'}+Ny')$ implies $N(x+y)=N(x'+y')$, and thus $x+y=x'+y'$ as $N\neq 0$. It follows that $$\varphi_A(\alpha_i+Nx)+\varphi_B(\beta_j+Ny)=(x+y,i+j)=(x'+y',i'+j')=
\varphi_A(\alpha_{i'}+Nx')+\varphi_B(\beta_{j'}+Ny').$$ Conversely, if $\varphi_A(\alpha_i+Nx)+\varphi_B(\beta_j+Ny)=
\varphi_A(\alpha_{i'}+Nx')+\varphi_B(\beta_{j'}+Ny')$, then $(x+y,i+j)=(x'+y',i'+j')$ follows, implying $x+y=x'+y'$ and $i+j=i'+j'$. Hence \eqref{def-alphbeta} ensures $\alpha_i+\beta_j=\alpha_{i'}+\beta_{j'}$, and now $$(\alpha_i+Nx)+(\beta_j+Ny)=\alpha_{i}+\beta_{j}+N(x+y)=\alpha_{i'}+\beta_{j'}+N(x'+y')=
(\alpha_{i'}+Nx')+(\beta_{j'}+Ny').$$ This shows that $A+B$ is Freiman isomorphic to the sumset $\varphi_A(A)+\varphi_B(B)= \bigcup_{i=0}^{m-1}(X_i\times\{i\})+\bigcup_{j=0}^{n-1}(Y_j\times \{j\})\subseteq \Z^2$, completing the proof.
\end{proof}

\begin{lemma}\label{lemma-2D-redcalc}
Let $x\geq 1$ and $y\geq 3$ be integers  and let $s\geq 1$ be the integer with  $$(s-1)s(\frac{y}{2}-1)+s-1< x\leq s(s+1)(\frac{y}{2}-1)+s.$$ Then $$\min\left\{\left\lceil(\frac{x}{m}+\frac{y}{n}-1)(m+n-1)\right\rceil:\; m,n\in \Z,\,x\geq m\geq 1,\,\frac{y}{3}+1\geq n\geq 2\right\}=\left\lceil(\frac{x}{s}+\frac{y}{2}-1)(s+1)\right\rceil.$$
\end{lemma}

\begin{proof}
Assuming the lemma fails, we obtain
\be\label{failhyp} (\frac{x}{m}+\frac{y}{n}-1)(m+n-1)-(\frac{x}{s}+\frac{y}{2}-1)(s+1)+\frac{1}{s}\leq 0\ee for some integers $m\geq 1$ and  $n\geq 2$ with $y\geq 3n-3$ and $x\geq m$ (note $(\frac{x}{s}+\frac{y}{2}-1)(s+1)$ can be expressed as a rational fraction with denominator $s$ regardless of the parity of $s$). Multiplying \eqref{failhyp} by $2smn$ yields
\be\label{failhyper} 2n(s(n-1)-m)x+sm(2m-2-(s-1)n)y-2smn(m+n-s-2)+2mn\leq 0\ee

\subsection*{Case 1:} $n=2$.

\begin{proof}
In this case, \eqref{failhyper} yields $2(s-m)x\leq sm(s-m)y-2sm(s-m)-2m$, implying $m\neq s$. If $m\leq s-1$, then  we obtain $x\leq sm(y/2-1)-\frac{m}{s-m}$. Considering this upper bound as a function of $m$, we find that its discrete derivative (its value  at $m+1$ minus its value at $m$) equals $s(\frac{y}{2}-1-\frac{1}{(s-m)(s-m-1)})\geq 0$ (for $m\leq s-2$), meaning it is maximized when $m$ achieves the upper bound $m=s-1$, yielding $x\leq s(s-1)(y/2-1)-s+1$, contrary to hypothesis. On the other hand, if $m\geq s+1$, then we obtain $x\geq sm(y/2-1)+\frac{m}{m-s}$. Considering this lower bound as a function of $m$,  we find that its discrete derivative (its value  at $m+1$ minus its value at $m$) equals $s(\frac{y}{2}-1-\frac{1}{(m-s)(m+1-s)})\geq 0$ (for $m\geq s+1$), meaning it is minimized when $m$ achieves the lower bound $m=s+1$, yielding $x\geq s(s+1)(y/2-1)+s+1$, contrary to hypothesis, completing the case.
\end{proof}

In view of Case 1, we now assume $n\geq 3$.

\subsection*{Case 2:} $s(n-1)\geq m$.

\begin{proof}

In this case, the coefficient of $x$ in \eqref{failhyper} is non-negative.

\medskip

Suppose first that $s=1$, in which case the coefficient of $y$ in \eqref{failhyper} is also non-negative. Thus using the estimates $x\geq m$ and $y\geq 3n-3$ in \eqref{failhyper}, followed by the estimate $n\geq 3$ (in view of Case 1), yields the contradiction (dividing all terms by $2m$)
$$0\geq nm-3m+3\geq 3.$$ So we now assume $s\geq 2$.

\medskip

 As the coefficient of $x$ in \eqref{failhyper} is non-negative, applying  the hypothesis $x\geq s(s-1)(y/2-1)+s$ yields
\be\label{hum}\Big(s(s-1)n^2-(s-1)(s+2m)n+2m^2-2m\Big)y-2(s^2-2s+m)n^2-2(m^2-2sm-\frac{m}{s}
-s^2+2s)n\leq 0.\ee
%
%$x\geq s(s-1)(y/2-1)+1$  and multiplying by $2mn$, we obtain
%\be\label{humII}\Big(s(s-1)n^2-(s-1)(s+2m)n+2m^2-2m\Big)y-2(s^2-s+m-1)n^2-2(m^2-2sm-m-s^2+s+1+\frac{m}{s})n<0.\ee

We next need to show that the coefficient of $y$ in \eqref{hum} is non-negative. To this end, assume by contradiction that \be\label{ycoeff}s(s-1)n^2-(s-1)(s+2m)n+2m^2-2m<0.\ee Since $m$ and $s$ are positive integers, \eqref{ycoeff} fails  for $s=1$, allowing us to assume $s\geq 2$.
Thus \eqref{ycoeff} is quadratic in $n$ with positive lead coefficient. The expression in \eqref{ycoeff} has non-negative derivative for  $n\geq \frac{s+2m}{2s}$.
Consequently, since our case hypothesis gives $n\geq\frac{m}{s}+1>\frac{s+2m}{2s}$, we conclude that the derivative with respect to $n$ in \eqref{ycoeff} is non-negative.
In particular,  \eqref{ycoeff} must hold with $n=\frac{m+s}{s}$, yielding $$(s+1)m(m-s)<0.$$
Thus $m\leq s-1$.
Since the derivative with respect to $n$ in \eqref{ycoeff} is non-negative for $n\geq \frac{s+2m}{2s}$ and
$n\geq 2>\frac{s+2m}{2s}$ (as $m\leq s-1$), it follows that \eqref{ycoeff} must also hold for $n=2$, yielding
$$2(m-s)(m-s+1)<0,$$ which contradicts that $m\leq s-1$. So we conclude that \eqref{ycoeff} fails, meaning the coefficient of $y$ in \eqref{failhyper} is non-negative.

As a result, applying the hypothesis $y\geq 3n-3$  in \eqref{hum}   yields \be\label{humdull}(4n-6)m^2-(n^2(6s-4)-(10s-12+\frac{2}{s})n-6)m+sn(n-1)(3(s-1)n-5s+7)\leq 0.\ee
The above expression is quadratic in $m$ with positive lead coefficient $4n-6>0$ (as $n\geq 2$) and discriminant  equal to $4$ times the quantity
\be\label{discr}-n(n-2)(n-3)(3n-5)s^2-2n(n-2)(n-3)s+(4n^4-30n^3+58n^2-36n+9)+\frac{n^2+s(4n^3
-12n^2+6n)}{s^2}\ee
Since $n\geq 3$ is an integer, the derivative with respect to $s$ of \eqref{discr} is negative, meaning   \eqref{discr} is maximized for $s=2$, in which case it equals
$-8n^4+48n^3-100n^2+63n+\frac14 n^2+9$, which is negative for $n\geq 2$ (it has two complex roots with largest real root less than $2$). Thus the discriminant of \eqref{humdull} is negative for $s\geq 2$, contradicting that \eqref{humdull} is non-positive, which completes Case 2.
\end{proof}

\subsection*{Case 3:} $s(n-1)<m$.

\begin{proof}
In this case, the coefficient of $x$ in \eqref{failhyper} is negative, so we can  apply the estimate $x\leq s(s+1)(y/2-1)+s$ to yield
\be\label{humII}\Big(s(s+1)n^2-s(s+2m+1)n+2m^2-2m\Big)y-2(s^2+m)n^2+2
(s^2+2sm-m^2+2m+\frac{m}{s})n\leq 0.\ee

We next need to show that the coefficient of $y$ in \eqref{humII} is non-negative. To this end, assume by contradiction that \be\label{ycoeffII}2m^2-(2sn+2)m+s(s+1)n(n-1)=s(s+1)n^2-s(s+2m+1)n+2m^2-2m<0.\ee
Considering \eqref{ycoeffII} as a function of $m$, we find that it has positive derivative when $m\geq \frac{sn+1}{2}$. Thus, since $m>s(n-1)\geq \frac{sn+1}{2}$ by case hypothesis (in view of $n\geq 3$), we see that \eqref{ycoeff} is minimized when $m=s(n-1)$, yielding $$(n-1)(n-2)s(s+1)<0,$$ which fails in view of $s\geq 1$ and $n\geq 2$. So we instead conclude that the coefficient of $y$ in \eqref{humII} is non-negative.

As a result, applying the hypothesis $y\geq 3n-3$  in \eqref{humII}   yields \be\label{humduller}(4n-6)m^2-(6sn^2+2n^2-10sn+2n-6-\frac{2n}{s})m+sn(3sn^2+3n^2-8sn-6n+5s+3)\leq 0.\ee
The above expression is quadratic in $m$ with positive lead coefficient $4n-6>0$ (as $n\geq 2$) and discriminant  equal to $4$ times the quantity
\begin{align}\nn-n(n-2)(n-3)(3n-5)s^2-2n(n-2)(n-3)(3n-4)s+(n^4-4n^3+5n^2-6n+9)\\+
\frac{n^2+s(6n-2n^2-2n^3)}{s^2}\nn\\ <-n(n-2)(n-3)(3n-5)s^2-2n(n-2)(n-3)(3n-4)s+(n^4-4n^3+5n^2-6n+9)\label{aling}\end{align}
Since $n\geq 3$ is an integer, the derivative with respect to $s$ of \eqref{aling} is non-positive, meaning   \eqref{aling} is maximized for $s=1$, in which case it equals
$-8n^4+54n^3-114n^2+72n+9$, which is negative for $n\geq 4$ (it has two complex roots with largest real root less than $4$). Thus the discriminant of \eqref{humduller} is negative for $n\geq 4$, contradicting that \eqref{humduller} is non-positive. It remains only to consider the case when $n=3$.

For $n=3$, \eqref{humduller} becomes  (dividing all terms by $6$)
\be\label{humdullest}m^2-(4s+3-\frac{1}{s})m+s(4s+6)\leq 0.
\ee
By case hypothesis, $m\geq (n-1)s+1=2s+1$, while \eqref{humdullest} is minimized for $m=2s+1+\frac12-\frac{1}{2s}$. Thus, since $m$ is an integer, we see \eqref{humdullest} is minimized when $m=2s+1$, in which case \eqref{humdullest} yields the contradiction $1/s\leq 0$, which is a proof concluding contradiction.
\end{proof}
\end{proof}

The following proposition gives a rough estimate for the resulting bound from Lemma \ref{lemma-2D-redcalc}.

\begin{proposition}
\label{prop-rough-estimate} For real numbers $x,y,s> 0$ with $y>2$, we have
$$(\frac{x}{s}+\frac{y}{2}-1)(s+1)\geq x+\frac{y}{2}-1+2\sqrt{x(\frac{y}{2}-1)}.$$
\end{proposition}

\begin{proof}
We have $(\frac{x}{s}+\frac{y}{2}-1)(s+1)=x+\frac{y}{2}-1+\frac{x}{s}+s(\frac{y}{2}-1)$. Thus, if the proposition fails, then $0<\frac{2x}{s}+(y-2)s<\sqrt{8x(y-2)}$. Multiplying by $s$ and squaring both sides, we obtain $4x^2+(y-2)^2s^4+4s^2x(y-2)<8s^2x(y-2)$, implying $$0>4x^2+(y-2)^2s^4-4s^2x(y-2)=(2x-(y-2)s^2)^2,$$
which is not possible.
\end{proof}

We now proceed with the proof of our main result.

\begin{proof}[Proof of Theorem \ref{thm-3k-4-minimprove}]  %If $|B|\leq 2$, then $B$ is trivially itself an arithmetic progression, in which case the theorem holds talking $B=P_B$ in view of Theorem \ref{CDT-for-Z}. Therefore we may assume $|B|\geq 3$.
We may  w.l.o.g. assume $0=\min A=\min B$ and $\gcd(A+B)=1$. In view of \eqref{hyp1a}, we have
\be\nn |A+B|<|A|+\frac{|A|}{s}+\frac{s+1}{2}|B|-s-1.\ee

%\medskip

%If $s=1$, then \eqref{hyp-bounds} becomes $|A|<|B|$, and \eqref{hyp1a} becomes $|A+B|\leq 2|A|+|B|-3$, in which case the desired conclusion follows from the $3k-4$ Theorem (Theorem \ref{}). Therefore we may assume $s\geq 2$.

\medskip

Let us begin by showing it suffices to prove the theorem in the case $\gcd^*(B)=1$, that is, when $B-B$ generates $\la A+B\ra_*=\Z$. To this end, assume we know the theorem is true when $\gcd^*(B)=1$ but $\gcd^*(B)=d\geq 2$.
We can partition $A=A_1\cup A_2\cup\ldots\cup A_t$ with each $A_i$ a maximal nonempty subset of elements congruent to each other modulo $d$. For $i\in [1,t]$, let $s_i\geq 1$ be the integer with $$(s_i-1)s_i(|B|/2-1)+s_i-1<|A_i|\leq s_i(s_i+1)(|B|/2-1)+s_i.$$ Note that $\gcd^*(A_i+B)=d=\gcd^*(B)$ for every $i\in [1,t]$. Thus, if $|A_i+B|<(\frac{|A_i|}{s_i}+\frac{|B|}{2}-1)(s_i+1)$ for some $i\in [1,t]$, then we could apply the case $\gcd^*(B)=1$ to the sumset $A_i+B$ (since $B-B$ generates $d\Z=\la A_i+B\ra_*$) thereby obtaining the desired conclusion for $B$. Therefore, we can instead assume this fails, meaning
\be|A_i+B|\geq(\frac{|A_i|}{s_i}+\frac{|B|}{2}-1)(s_i+1)=|A_i|+\frac{|A_i|}{s_i}+\frac{s_i+1}{2}|B|
-s_i-1\quad\mbox{ for every $i\in [1,t]$}.\label{weetee}\ee
Since the sets $A_i$ are distinct modulo $d$ with $B\subseteq d\Z$, it follows that the sets $A_i+B$ are disjoint for $i\in [1,t]$. Thus
\be\label{weevee}|A+B|\geq \Sum{i=1}{t}|A_i+B|\geq \Sum{i=1}{t}\left(|A_i|+\frac{|A_i|}{s_i}+\frac{s_i+1}{2}|B|-s_i-1\right),\ee with the latter inequality in view of \eqref{weetee}.
Let $m=s_1+\ldots+s_t$. Note $|A_1|+\ldots+|A_t|=|A|$ and  $1\leq s_i\leq |A_i|$ for all $i\in [1,t]$ (in view of the definition of $s_i$). Thus $1\leq t\leq m\leq |A|$.  A simple inductive argument on $t$ (with base case $t=2$) shows that $\Sum{i=1}{t}\frac{x_i}{y_i}\geq \left(\Sum{i=1}{t}x_i\right)/\left(\Sum{i=1}{t}y_i\right)$ holds for any positive real numbers $x_1,y_1,\ldots,x_t,y_t>0$. In particular, $\Sum{i=1}{t}\frac{|A_i|}{s_i}\geq \left(\Sum{i=1}{t}|A_i|\right)/\left(\Sum{i=1}{t}s_i\right)=\frac{|A|}{m}$.
Applying this estimate in \eqref{weevee}, along with the identities $|A_1|+\ldots+|A_t|=|A|$ and $m=s_1+\ldots+s_t$, yields \begin{align}|A+B|\geq \nn |A|+\frac{|A|}{m}+\frac{m}{2}|B|-m+t(|B|/2-1)&\geq |A|+\frac{|A|}{m}+\frac{m}{2}|B|-m+|B|/2-1\\&=(\frac{|A|}{m}+\frac{|B|}{2}-1)(m+1)\label{kaydey}.
\end{align}
Since $1\leq m\leq |A|$, \ $|B|\geq 3$ and $2\leq \frac{|B|}{3}+1$, Lemma \ref{lemma-2D-redcalc} (applied with $x=|A|$, $y=|B|$, and $n=2$) implies $\lceil(\frac{|A|}{m}+\frac{|B|}{2}-1)(m+1)\rceil\geq (\frac{|A|}{s}+\frac{|B|}{2}-1)(s+1)$. As a result, since $|A+B|$ is an integer, we see that \eqref{kaydey} yields the lower bound $|A+B|\geq (\frac{|A|}{s}+\frac{|B|}{2}-1)(s+1)$, contrary to hypothesis. So it remains to prove the theorem when $\gcd^*(B)=1$, which we now assume.

\medskip

We proceed by induction on $|A|$. Note, if $|A|=1$, then  $s=1$ and the bound $|A+B|\geq |B|=(\frac{|A|}{s}+\frac{|B|}{2}-1)(s+1)$ holds trivially. This completes the base of the induction and allows us to assume $|A|\geq 2$.

\medskip

Suppose $\gcd^*(A)=d>1$.  Then $A$ is contained in a $d\Z$-coset. In view of $\gcd^*(B)=1$ and $d\geq 2$, it follows that there are $t\geq 2$ $d\Z$-coset representatives $\beta_1,\ldots,\beta_t\in \Z$ such that each slice $B_{\beta_i}=(\beta_i+\Z)\cap B$ is nonempty for $i\in [1,t]$. Applying Theorem \ref{CDT-for-Z} to $A+B_{\beta_i}$ for each $i\in [1,t]$ yields $|A+B|\geq \Sum{i=1}{t}(|A|+|B_{\beta_i}|-1)=t(|A|-1)+|B|\geq 2|A|+|B|-2\geq(\frac{|A|}{s}+\frac{|B|}{2}-1)(s+1)$, with the final inequality in view of Lemma \ref{lemma-2D-redcalc} (applied with $x=|A|$, $y=|B|$, $m=1$ and $n=2$), contrary to hypothesis. So we instead conclude that $${\gcd}^*(A)={\gcd}^*(B)=1.$$

\medskip

By translation, we may assume $B\subseteq [0,n]$ and $A\subseteq [0,m]$ with $0,\,n\in B$ and $0,\,m\in A$. Define $P_B:=[0,n]$. Let $\phi_n:\Z\rightarrow \Z/n\Z$ be the reduction modulo $n$ homomorphism and set $G=\Z/n\Z$.
We aim to use modular reduction as described above Corollary \ref{cor-modred}. To that end, let $\wtilde A$ and $\wtilde B$, as well as all associated notation,  be defined as above Corollary \ref{cor-modred} using the modulus $n=\max B-\min B$.
In particular, $A_t\subseteq \ldots\subseteq A_0=\phi_n(A)\subseteq \Z/n\Z$, where $t\geq 0$ is the maximal index such that  $A_t\neq \emptyset$,
$B_1=\{0\}$,  $B_0=\phi_n(B)\subseteq \Z/n\Z$,  $|B_0|=|B|-1$, \ $\Sum{i=0}{t}|A_i|=|A|$, and  $$n=|P_B|-1.$$
%Now $\wtilde A+\wtilde B=\bigcup_{i=0}^{t+1}(C_i\times \{i\})$ with
 %$C_0=A_0+B_0$, \ $C_{t+1}=A_t+B_1=A_t$ and $$C_i=(A_i+B_0)\cup (A_{i-1}+B_1)=(A_i+B_0)\cup A_{i-1}\quad\mbox{ for $i\in [1,t]$}.$$
%If $H\leq G=\Z/n\Z$ is a subgroup, and $z\in G\setminus (H+A_0)$, then

\subsection*{Case 1:} $A_0+B_0=\Z/n\Z$.

\begin{proof}
In this case, Corollary \ref{cor-modred} implies that $|A|+|B|+r=|A+B|\geq |A_0+B_0|+|A|=n+|A|$, implying $|P_B|=n+1\leq |B|+1+r$, as desired.
\end{proof}

%In view of Case 1,  we may assume $|A_0+B_0|<n$, which implies $|A_i+B_0|<n$ for all $i\in [0,t]$ in view of $A_t\subseteq \ldots\subseteq A_0$.

\subsection*{Case 2:} $|A_0+B_0|<\min\{n,\,|A_0|+|B_0|-1\}$.

\begin{proof}
Let $H=\mathsf H(A_0+B_0)\leq G$. In view of the case hypothesis,  Kneser's Theorem (Theorem \ref{thm-kt}) implies that $H$ is a \emph{proper}, nontrivial subgroup of $G=\Z/n\Z$ with $|A_0+B_0|\geq |H+A_0|+|H+B_0|-|H|$ and \be\label{kst-hyp}|\phi_H(A_0)+\phi_H(B_0)|=|\phi_H(A_0)|+|\phi_H(B_0)|-1<|G/H|.\ee Note $\phi_H(A_0)+\phi_H(B_0)$ is aperiodic as $H=\mathsf H(A_0+B_0)$ is the maximal period of $A_0+B_0$, and \be\label{kt-holes}|(H+A_0)\setminus A_0|+|(H+B_0)\setminus B_0|\leq |H|-2,\ee else $|A_0+B_0|\geq |A_0|+|B_0|-1$ (in view of the bound from Kneser's Theorem), contrary to case hypothesis.
In view of \eqref{kst-hyp} and $G/H$ being nontrivial (as $H<G$ is proper), we can apply the Kemperman Structure Theorem (Theorem \ref{KST-}) to $\phi_H(A_0)+\phi_H(B_0)$. Then there exists a proper subgroup $L<G$ with $H\leq L$ such that $(\phi_L(A_0),\phi_L(B_0))$ is an elementary pair of some type (I)--(IV). Indeed, if type (IV) occurs, then $L=H$. Moreover, for types (I)--(III), there exist nonempty $L$-coset slices $A_\emptyset\subseteq A_0$ and $B_\emptyset\subseteq B_0$ inducing $L$-quasi-periodic decompositions in $H+A$ and $H+B$, so  $H+(A_0\setminus A_\emptyset)$ and $H+(B_0\setminus B_\emptyset)$ are both $L$-periodic, $\phi_H(A_\emptyset)+\phi_H(B_\emptyset)\in \phi_H(A)+\phi_H(B)$ is a unique expression element, and $$|A_\emptyset+B_\emptyset|=|H+A_\emptyset|+|H+B_\emptyset|-|H|.$$
 %with $A_\emptyset$ and $B_\emptyset$ each contained in an $L$-coset.

\subsection*{Subcase 2.1:} $(\phi_L(A_0),\phi_L(B_0))$ has type (I).

In this case, either $|\phi_L(A_0)|=1$ or $|\phi_L(B_0)|=1$, both contradicting that $\gcd^*(A)=\gcd^*(B)=1$ in view of $L<G=\Z/n\Z$ being a \emph{proper} subgroup.

 \subsection*{Subcase 2.2:}  $(\phi_L(A_0),\phi_L(B_0))$ has type (IV).

In this case, $H=L$, $|\phi_H(A_0)|,\,|\phi_H(B_0)|\geq 3$, every element in $\phi_H(A_0)+\phi_H(B_0)$ has at least $2$ representations, and $$|A_0+B_0|=|G|-|H|.$$ Since $|\phi_H(A_0)+\phi_H(B_0)|=|\phi_H(A_0)|+|\phi_H(B_0)|-1\geq |\phi_H(A_0)|+2$, it follows that there are two distinct $H$-cosets $\gamma_1+H$ and $\gamma_2+H$ which intersect $A_0+B_0$ but not $A_0$. For each $\gamma_i$, we can find $\alpha_i\in A_0$ and $\beta_i\in B_0$ such that $\gamma_i+H=\alpha_i+\beta_i+H$, and we choose the pair $(\alpha_i,\beta_i)$ to maximize $|A_0\cap (\alpha_i+H)|+|B_0\cap (\beta_i+H)|$. Since every element in  $\phi_H(A_0)+\phi_H(B_0)$ has at least $2$ representations, it follows from the pigeonhole principle and \eqref{kt-holes} that $$|A_0\cap (\alpha_i+H)|+|B_0\cap (\beta_i+H)|\geq 2|H|-\frac12(|H|-2)=\frac32|H|+1\quad \mbox{ for $i=1,2$}.$$
Since each $\gamma_i+H$ does not intersect $A_0=A_0+B_1$, it follows from Corollary \ref{cor-modred} that \begin{align*}|A|+|B|+r=|A+B|&\geq |A_0+B_0|+|A|+2(\frac32|H|+1-1-|H|)\\ &=|A_0+B_0|+|A|+|H|=|G|+|A|=n+|A|,\end{align*} implying $|P_B|=n+1\leq |B|+r+1$, as desired.

\subsection*{Subcase 2.3:}  $(\phi_L(A_0),\phi_L(B_0))$ has type (III).

In this case, $|\phi_L(A_0)|,\,|\phi_L(B_0)|\geq 3$ and $$|A_0+B_0|=|(A_0+B_0)\setminus (A_\emptyset+B_\emptyset)|+|A_\emptyset+B_\emptyset|=(|G|-|L|)+
(|H+A_\emptyset|+|H+B_\emptyset|-|H|).$$ Moreover, by Lemma \ref{thm-typeIII}, we have \ \be\label{steefel}\phi_L(A_0\setminus A_\emptyset)+\phi_L(B_0\setminus B_\emptyset)=\phi_L(A_0+B_0)\setminus \phi_L(A_\emptyset+B_\emptyset).\ee Since $|\phi_L(A_0)+\phi_L(B_0)|=|\phi_L(A_0)|+|\phi_L(B_0)|-1\geq |\phi_L(A_0)|+2$, it follows that there is some $L$ coset  $\gamma+L$  that intersects $A_0+B_0$ but not $A_0$ and which is distinct from the $L$-coset $A_\emptyset+B_\emptyset+L$. Then \eqref{steefel} ensures there are $\alpha\in A_0\setminus A_\emptyset$ and $\beta\in B_0\setminus B_\emptyset$ with $\alpha+\beta+L=\gamma+L$. As a result, since  $H+(A_0\setminus A_\emptyset)$ and $H+(B_0\setminus B_\emptyset)$ are both $L$-periodic, it follows that  $$|A_0\cap (\alpha+L)|+|B\cap (\beta+L)|\geq 2|L|-(|(H+A_0)\setminus A_0|+|(H+B_0)\setminus B_0)|)\geq 2|L|-|H|+2,$$ with the final inequality in view of \eqref{kt-holes}.
Since $\gamma+L$ does not intersect $A_0$, it follows from Corollary \ref{cor-modred} that \begin{align*}|A|+|B|+r=|A+B|&\geq |A_0+B_0|+|A|+(2|L|-|H|+2-|L|-1)\\ &=|A_0+B_0|+|A|+|L|-|H|+1\\&=(|G|-|L|+|H+A_\emptyset|+|H+B_\emptyset|-|H|)+|A|+|L|-|H|+1\\&\geq |G|+|A|+1=n+1+|A|,\end{align*} implying $|P_B|=n+1< |B|+r+1$, as desired.

\subsection*{Subcase 2.4:}  $(\phi_L(A_0),\phi_L(B_0))$ has type (II).

In this case, Lemma \ref{Lemma-ap-mod-reduction} implies that $A+B$ is Freiman isomorphic to a sumset $A'+B'\subseteq \Z^2$ with $B'$ contained in exactly $n'=|\phi_L(B_0)|\geq 2$ lines parallel to the horizontal axis, and $A'$ contained in exactly $m'=|\phi_L(A_0)|\geq 2$ lines parallel to the horizontal axis. Let $x=(1,0)$ and $y=(0,1)$. Compressing along the horizontal axis results in a sumset $A''+B''\subseteq \Z^2$, where $A''=\mathsf C_{x,y}(A')$ and  $B''=\mathsf C_{x,y}(B')$. Then $|A+B|=|A'+B'|\geq |A''+B''|$,  $|A''|=|A'|=|A|$ and $|B''|=|B'|=|B|$.
Since $H+(A_0\setminus A_\emptyset)$ and $H+(B_0\setminus B_\emptyset)$ are both $L$-periodic with $A_\emptyset\subseteq A_0$ and $B_\emptyset\subseteq B_0$ each $L$-coset slices, it follows from \eqref{kt-holes} that \begin{align*}|(L+B_0)\setminus B_0|&=|(L+B_0)\setminus (H+B_0)|+|(H+B_0)\setminus B_0|\\&= (|L|-|H+B_\emptyset|)+|(H+B_0)\setminus B_0|\leq |L|-|H|+|H|-2=|L|-2.\end{align*} Thus  $$|B|=|B_0|+1\geq n'|L|-|L|+3.$$ As a result, if $|L|\geq 3$, then $|B''|=|B|\geq 3n'$, in which case Theorem \ref{2D-brunn-Mink} (applied with $\ell=\R x$) and Lemma \ref{lemma-2D-redcalc} (applied with  $m=m'$, $n=n'$, $x=|A|=|A''|$ and $y=|B|=|B''|$) imply $|A+B|\geq |A''+B''|\geq (\frac{|A|}{s}+\frac{|B|}{2}-1)(s+1)$, contrary to hypothesis. Likewise, if $|L|=2$ and $n'=2$, then $|B''|=|B|\geq 2|L|-|L|+3=5\geq 3n'-3$, whence Theorem \ref{2D-brunn-Mink} (applied with $\ell=\R x$)  and Lemma \ref{lemma-2D-redcalc} (applied with  $m=m'$, $n=2$, $x=|A|=|A''|$ and $y=|B|=|B''|$) again yield the contradiction  $|A+B|\geq |A''+B''|\geq (\frac{|A|}{s}+\frac{|B|}{2}-1)(s+1)$.
We are left to consider the case when $|L|=2$ and $n'\geq 3$, in which case $|B''|=|B|\geq n'|L|-|L|+3=2n'+1\geq 7$. Each horizontal line that intersects $B''$ contains at most $|L|+1\leq 3$ elements (as $B=B_0\cup B_1$ with $|B_1|=1$ and the elements of $B_0$ distinct modulo $n$), ensuring via the definition of compression that $B''$ is contained in $n''\leq 3$ vertical lines. Note  $|B|\geq n'|L|-|L|+3=2n'+1>n'$ ensures some horizontal line has at least two elements, whence $n''\geq 2$.  Thus Theorem \ref{2D-brunn-Mink} (applied with $\ell=\R y$) and Lemma \ref{lemma-2D-redcalc} (applied with  $n=n''\in [2,3]$, $x=|A|=|A''|$ and $y=|B|=|B''|$,  noting that $|B''|=|B|\geq 7$ ensures $3n'-3\leq 6<7\leq |B|$) again yields the contradiction $|A+B|\geq |A''+B''|\geq (\frac{|A|}{s}+\frac{|B|}{2}-1)(s+1)$, completing Case 2.
\end{proof}

\subsection*{Case 3:} $|A_0+B_0|\geq |A_0|+|B_0|-1$.

\begin{proof}
Decompose $A=\bigcup_{i=1}^{|A_0|}X_i$, \ $B=\bigcup_{j=1}^{|B_0|}Y_j$ and $A+B=\bigcup_{i=1}^{|A_0|}\bigcup_{j=1}^{|B_0|}(X_i+Y_j)=\bigcup_{k=1}^{|A_0+B_0|}Z_k$ modulo $n$, where the $X_i\subseteq A$ are the maximal nonempty subsets of elements congruent modulo $n$, and likewise for the $Y_j\subseteq B$ and $Z_k\subseteq A+B$. For $i\in [1,|A_0|]$, let $X'_i$ be obtained from $X_i$ by removing the smallest element from $X_i$. Set $A'=\bigcup_{i=1}^{|A_0|}X'_i$ and decompose $A'+B=\bigcup_{k=1}^{|A_0+B_0|}Z'_k$ with the $Z'_k\subseteq Z_k$ (possibly empty). Each $X'_i+Y_j\subseteq X_i+Y_j$ is missing the smallest element of $X_i+Y_j$, as this was a unique expression element in $X_i+Y_j$. As a result, since each $Z_k$ is a union of sets of the form $X_i+Y_j$, it follow that each $Z'_k\subseteq Z_k$ is missing the smallest element of $Z_k$. In consequence,
\be \label{tangent2}|A'|=|A|-|A_0|\quad\und\quad |A'+B|\leq |A+B|-|A_0+B_0|\leq |A+B|-|A_0|-|B|+2,\ee
with the final inequality above in view of $|B_0|=|B|-1$ and the case hypothesis.

\medskip

If $|A|=|A_0|$, then Theorem \ref{modular-red-cor} and the case hypothesis imply that $|A+B|\geq |\wtilde A+\wtilde B|=|A_0+B_0|+|A_0+B_1|=|A_0+B_0|+|A_0|\geq 2|A_0|+|B_0|-1=2|A|+|B|-2$, while $2|A|+|B|-2\geq(\frac{|A|}{s}+\frac{|B|}{2}-1)(s+1)$ follows by Lemma \ref{lemma-2D-redcalc} (applied with $x=|A|$, $y=|B|$, $m=1$ and $n=2$), yielding $|A+B|\geq (\frac{|A|}{s}+\frac{|B|}{2}-1)(s+1)$, which is contrary to hypothesis. Therefore we instead conclude that $|A_0|<|A|$, ensuring that $A'$ is nonempty.

\medskip

Let $s'\geq 1$ be the integer such that \be\label{wiggleworm}(s'-1)s'\left(\frac{|B|}{2}-1\right)+s'-1<|A'|\leq s'(s'+1)\left(\frac{|B|}{2}-1\right)+s'.\ee
Note, since $|A'|<|A|$, that $s'\leq s$.
If $|A'+B|<(\frac{|A'|}{s}+\frac{|B|}{2}-1)(s+1)$, then applying the induction hypothesis to $A'+B$ yields the desired conclusion for $B$. Therefore we can assume
$$|A'+B|\geq (\frac{|A'|}{s'}+\frac{|B|}{2}-1)(s'+1).$$ Combined with \eqref{tangent2}, we find
\begin{align} |A+B|&\geq (\frac{|A|-|A_0|}{s'}+\frac{|B|}{2}-1)(s'+1)+|A_0|+|B|-2\nn\\
&=|A|+\frac{|A|}{s'}+\frac{s'+3}{2}|B|-s'-3-\frac{|A_0|}{s'}.\label{gocu}\end{align}

Now Corollary \ref{cor-modred} and the case hypothesis imply $|A+B|\geq|A_0+B_0|+|A|\geq |A_0|+|B_0|-1+|A|=|A|+|B|-2+|A_0|$. Combined with the hypothesis $|A+B|<(\frac{|A|}{s}+\frac{|B|}{2}-1)(s+1)$, we conclude that \be\label{A0-small} |A_0|<\frac{|A|}{s}+(s-1)(\frac{|B|}{2}-1).\ee

\medskip

\subsection*{Subcase 3.1} $1\leq s'\leq s-2$.

In this case,  $s\geq 3$ and \eqref{wiggleworm} gives  $|A|-|A_0|=|A'|\leq (s-2)(s-1)(|B|/2-1)+s-2$, which combined with \eqref{A0-small} yields $\frac{s-1}{s}|A|-(s-1)(\frac{|B|}{2}-1)<(s-2)(s-1)(\frac{|B|}{2}-1)+s-2$, in turn implying $$|A|<s(s-1)(\frac{|B|}{2}-1)+\frac{s(s-2)}{s-1}<s(s-1)(\frac{|B|}{2}-1)+s.$$ However, this  contradicts the hypothesis $|A|\geq (s-1)s(\frac{|B|}{2}-1)+s$.

\subsection*{Subcase 3.2:} $s'=s$.

In this case, the bounds defining $s$ and $s'$ ensure $$|A_0|=|A|-|A'|\leq \Big(s(s+1)(|B|/2-1)+s\Big)-\Big(s(s-1)(|B|/2-1)+s\Big)=s(|B|-2).$$ Thus \eqref{gocu} implies
\begin{align*}|A+B|&\geq |A|+\frac{|A|}{s}+\frac{s+1}{2}|B|-s-1+|B|-2-\frac{|A_0|}{s}
\\ &\geq |A|+\frac{|A|}{s}+\frac{s+1}{2}|B|-s-1=(\frac{|A|}{s}+\frac{|B|}{2}-1)(s+1),
\end{align*} which is contrary to hypothesis.

\subsection*{Subcase 3.2:} $1\leq s'=s-1$.

In this case, $s\geq 2$, while \eqref{gocu} and \eqref{A0-small} yield
$$|A+B|>|A|+\frac{|A|}{s-1}+\frac{s+2}{2}|B|-s-2-\frac{|A|}{s(s-1)}-(\frac{|B|}{2}-1).$$
Combined with the hypothesis $|A+B|<(\frac{|A|}{s}+\frac{|B|}{2}-1)(s+1)=|A|+\frac{|A|}{s}+\frac{s+1}{2}|B|-s-1$, we conclude that
$$\frac{|A|}{s}=\frac{|A|}{s-1}-\frac{|A|}{s(s-1)}<\frac{|A|}{s},$$
which is not possible.
\end{proof}

As the above cases exhaust all possibilities, the proof is complete.
\end{proof}

\begin{proof}[Proof of Corollary \ref{cor-3k-4-minimprove}] For $|B|\leq 2$, we have $B=P_B$ being itself an arithmetic progression, with $|P_B|=|B|\leq |B|+r+1$ in view of Theorem \ref{CDT-for-Z}. For $|B|\geq 3$,
the result is an immediate consequence of Theorem \ref{thm-3k-4-minimprove} and Proposition \ref{prop-rough-estimate} (applied with $x=|A|$,\  $y=|B|$ and $s$ as defined in the statement of Theorem \ref{thm-3k-4-minimprove}).
\end{proof}

\section{Concluding Remarks}
As mentioned in the introduction, the bound $|P_B|\leq |B|+r+1$ is tight in Theorem \ref{thm-3k-4-minimprove}. However, the examples showing this bound to be tight  (including variations of that given in the introduction) require \emph{both} $A$ and $B$ to be contained in short arithmetic progressions. Thus a strengthening of Theorem \ref{thm-3k-4-minimprove}, where the bound on $|P_B|$ is improved when $A$ is not contained in a short arithmetic progression, is expected. Indeed, it might be hoped that $|P_A|$ could be reasonably bounded so long as there is no partition $A=A_0\cup A_1$ of $A$ into nonempty subsets with $A_0+B$ and $A_1+B$ disjoint.

\end{document}